\documentclass[a4paper,11pt,reqno]{amsart}

\usepackage{amsmath}
\usepackage{amsthm}
\usepackage{amssymb}
\usepackage{graphicx}
\usepackage{appendix}
\usepackage[pdftex]{hyperref}
\hypersetup{breaklinks=true,colorlinks=true,citecolor=green,urlcolor=red,linkcolor=blue,pdfpagemode=UseNone}

\theoremstyle{plain}

\newtheorem{theorem}{Theorem}[section]
\newtheorem{lemma}[theorem]{Lemma}

\newcommand{\eps}{\varepsilon}

\def\COMMENT#1{}
\let\COMMENT=\footnote

\def\E{\mathbb{E}}
\DeclareMathOperator{\Var}{Var}

\def\Markstrom{Markstr\"om}
\def\Umea{Ume\aa}

\begin{document}

\title{The range of thresholds for diameter 2 in random Cayley graphs}
\author{Demetres Christofides}
\author{Klas \Markstrom}
\date{\today}
\begin{abstract}
Given a group $G$, the model $\mathcal{G}(G,p)$ denotes the probability space of all Cayley graphs of $G$ where each element of the generating set is chosen independently at random with probability $p$.

Given a family of groups $(G_k)$ and a $c \in \mathbb{R}_+$ we say that $c$ is the threshold for diameter 2 for $(G_k)$ if for any $\varepsilon > 0$ with high probability $\Gamma \in \mathcal{G}(G_k,p)$ has diameter greater than 2 if $p \leqslant  \sqrt{(c - \eps)\frac{\log{n}}{n}}$ and diameter at most 2 if $p \geqslant  \sqrt{(c + \eps)\frac{\log{n}}{n}}$. In~\cite{CM2} we proved that if $c$ is a threshold for diameter 2 for a family of groups $(G_k)$ then $c \in [1/4,2]$ and provided two families of groups with thresholds $1/4$ and $2$ respectively.

In this paper we study the question of whether every $c \in [1/4,2]$ is the threshold for diameter 2 for some family of groups. Rather surprisingly it turns out that the answer to this question is negative. We show that every $c \in [1/4,4/3]$ is a threshold but a $c \in (4/3,2]$ is a threshold if and only if it is of the form $4n/(3n-1)$ for some positive integer $n$.
\end{abstract}

\maketitle

\section{Introduction}

Let us begin by recalling that given a group $G$ and a subset $S$ of $G$, the Cayley graph $\Gamma = \Gamma(G;S)$ of $G$ with respect to $S$ has the elements of $G$ as its vertex set and has an edge between $g$ and $h$ if and only if $hg^{-1} \in S$ or $gh^{-1} \in S$. We ignore any loops or multiple edges. In particular, whether $1 \in S$ or not is immaterial. Observe for example that $\Gamma$ is connected if and only if the set $S$ generates the group $G$. Throughout the paper, we will often refer to the set $S$ as the generating set of the graph $\Gamma$ irrespectively of whether it is a generating set for the group $G$ or not. 

The model $\mathcal{G}(G,p)$ is the probability space of all graphs $\Gamma(G;S)$ in which every element of $G$ is assigned to the set $S$ independently at random with probability $p$. We refer the reader to~\cite{CM2} for a discussion of some similarities and differences between this model and various other models of random graphs.


In this paper we continue the study of the diameter of random Cayley graphs that was initiated in~\cite{CM2}. Given a family of groups $(G_k)$ of orders $n_k$ with $n_k \to \infty$ as $k \to \infty$ and a $c \in \mathbb{R}_+$ we say that $c$ is the threshold for diameter 2 for $(G_k)$ if for any $\varepsilon > 0$ with high probability $\Gamma \in \mathcal{G}(G_k,p)$ has diameter greater than 2 if $p \leqslant  \sqrt{(c - \eps)\frac{\log{n}}{n}}$ and diameter at most 2 if $p \geqslant  \sqrt{(c + \eps)\frac{\log{n}}{n}}$. In~\cite{CM2} we proved that if $c$ is a threshold for diameter 2 for a family of groups $(G_k)$ then $c \in [1/4,2]$. Moreover this is best possible as $1/4$ is the threshold for the family of symmetric groups $S_n$ while $2$ is the threshold for the family of products of two-cycles $(C_2)^n$. 

Given the above result it is natural to ask whether for every $c \in [1/4,2]$ there is a family of groups for which it is a threshold for diameter $2$. Rather surprisingly, it turns out that the answer is negative. The aim of this paper is to give a characterisation of the values of $c$ that can appear as thresholds. 

\begin{theorem}\label{CompleteRange}
Let $c \in [1/4,2]$. Then $c$ is the threshold for diameter 2 for a family of groups $(G_k)$ if and only if $c \in [1/4,4/3] \cup \{\frac{4n}{3n-1}:n \in \mathbb{N}\}$.
\end{theorem}

We break the proof of Theorem~\ref{CompleteRange} into the following results.

\begin{theorem}\label{SmallCase}
Let $c \in [1/4,1/2)$. Then $c$ is the threshold for diameter 2 for the family $S_n \times C_m$, where $m = \lfloor n!^{\frac{2c}{1-2c}}\rfloor$.
\end{theorem}

\begin{theorem}\label{LargeCase}
Let $c \in [1/2,1)$. Then $c$ is the threshold for diameter 2 for the family $G = (C_2)^{n} \times (C_m)$ where $m = \lfloor 2^{\frac{(1-c)n}{c}}\rfloor$.
\end{theorem}

\begin{theorem}\label{LargerCase}
Let $c \in [1,4/3)$. Then $c$ is the threshold for diameter 2 for the family $G = (C_2)^{n} \times (D_{2m})$ where $m = \lfloor 2^{\frac{(4-3c)n}{3c}}\rfloor$ and $D_{2m}$ is the dihedral group of order $2m$.
\end{theorem}

\begin{theorem}\label{InterestingCase}
Suppose $c > 4/3$ is the threshold for diameter 2 for some family of groups. Then $c \in \{\frac{4n}{3n-1}:n \in \mathbb{N}\}$. Moreover, given a positive integer $n$, $\frac{4n}{3n-1}$ is the threshold for diameter 2 for the family $((C_2)^k \times D_{4n})_k$.
\end{theorem}

It is evident that Theorem~\ref{CompleteRange} follows from Theorems~\ref{SmallCase}-\ref{InterestingCase}. These results will be consequences of the more general Theorem~\ref{General} which enables us to determine the threshold for diameter~2 for a family of groups provided we have enough information about some specific structure of the groups of this family. 

As we mentioned this looks like a surprising result as one would expect that by taking `suitable combinations' of groups with different thresholds, one should be able to interpolate to get all possible values between the two thresholds. This is exactly what happens in Theorem~\ref{SmallCase}. Symmetric groups have threshold $1/4$, cyclic groups have threshold $1/2$, and by taking a suitable direct product we can get any threshold in the interval $[1/4,1/2]$. However there is a natural reason why this result is not so surprising after all. As one would expect the result depends on various group properties and there are many such properties for which you cannot interpolate. We are thinking here of properties which say that when the group is `too close to being abelian' then it actually is abelian. For example it is a trivial observation that if more than half of the elements of a group are central (i.e.~they commute with every other element) then the group is abelian and so all elements are central. (In fact the same holds even when more than a quarter of the elements are central.) In a similar manner, if more than $5/8$ of pairs of elements commute then all pairs commute~\cite{Gus,Rus} and if less than $7|G|/4$ pairs of elements are conjugate then no pair of distinct elements are conjugate~\cite{BBW}. In our situation the important group theoretic property which affects the final outcome is the percentage of involutions (i.e.~elements which square to the identity) of a finite group. It turns out that in the interval $[1/2,1]$ there are only finitely many values which can appear as proportion of involutions~\cite{Miller} and this is exactly what causes the set of thresholds in the interval $(4/3,2]$ to be discrete. 

The plan of the paper is as follows. In Section~2 we define certain dependency graphs of a group $G$. Their structure and more specifically the number of their edges will be the key factor to decide whether $c$ is the threshold for a family of groups or not. This is formulated precisely in Theorem~\ref{General}. The proof of this theorem will naturally split into two cases. The case where $p$ is below the threshold in which we show that the diameter of the random Cayley graph is with high probability greater than~$2$, and the case where $p$ is above the threshold in which we show that the diameter of the random Cayley graph is with high probability at most~$2$. We consider these cases in Sections~3 and~4 respectively. Finally in Section~5 we apply Theorem~\ref{General} to prove Theorems~\ref{SmallCase}-\ref{InterestingCase}. 

\section{The dependency graphs for distance greater than~2 from the identity}

Let $G$ be a finite group and let us denote its identity element by $1$. For each non-identity element $x$ of a group $G$, we define the graph $\Gamma_x$ whose vertices are the elements of $G$ and where $g$ is adjacent to $h$ if and only if $h \in \{xg,xg^{-1},x^{-1}g,x^{-1}g^{-1},gx,gx^{-1},g^{-1}x,g^{-1}x^{-1}\}$. Here we do allow loops but we ignore multiple edges.


The reason for the definition of this graph is the following: If $S$ is a subset of $G$ then $x$ is at distance greater than~2 from the identity in $\Gamma(G;S)$ precisely when $x,x^{-1} \notin S$ and $\Gamma_x[S]$ is independent. Indeed, $x$ is a neighbour of the identity if and only if $x \in S$ or $x^{-1} \in S$ and $x$ is at distance~2 from the identity if and only if $x,x^{-1} \notin S$ and furthermore there is a $y$ such that $1$ is a neighbour of $y$ and $y$ is a neighbour of $x$ in $\Gamma(G;S)$. The latter happens if and only if at least one of $y,y^{-1}$ and at least one of $yx^{-1},xy^{-1}$ belong to $S$. Setting $g$ to be equal to one of these four expressions and setting $h$ to be equal to one of the other two we get that this happens if and only if $h \in \{xg,xg^{-1},x^{-1}g,x^{-1}g^{-1},gx,gx^{-1},g^{-1}x,g^{-1}x^{-1}\}$. 

We now make a few simple observations about these dependency graphs
\begin{itemize}
\item[(O1)] $\Gamma_x$ has maximum degree at most $8$.
\item[(O2)] Every vertex of $\Gamma_x$ is incident to an edge which is not a loop.
\item[(O3)] $\Gamma_x$ has a loop at $y$ if and only if $x = y^2$ or $x = y^{-2}$.
\item[(O4)] For every two distinct elements $g,h$ of $G$ there are at most eight distinct $\Gamma_x$'s in which $g$ and $h$ are adjacent.
\end{itemize}

Given a graph $\Gamma$ with loops but no multiple edges we write $L(\Gamma)$ for the set of loops of $\Gamma$ and $E(\Gamma)$ for the set of edges of $\Gamma$ which are not loops. We also write $\ell(\Gamma)$ and $e(\Gamma)$ for the sizes of the corresponding sets.

For a group $G$, and an element $x$ of $G$ distinct from the identity we write
\[ g_G(x) = |G|^{-\frac{e(\Gamma_x)}{|G|}}.\]
Also, if $x^2=1$ then we set
\[ h_G(x) = e^{-\sqrt{\frac{\log{n}}{n}}\ell(\Gamma_x)}\]
while if $x^2 \neq 1$ then we set
\[ h_G(x) = e^{-2\sqrt{\frac{\log{n}}{n}}\ell(\Gamma_x)}.\]
Given a positive real number $c$ we write
\[ f_G(x;c) = g_G(x)^ch_G(x)^{\sqrt{c}}.\]
Finally, for a graph $G$ and a positive real number $c$ we write
\[ f_G(c) = \sum_{x \in G \setminus\{1\}} f_G(x;c).\]
As the group $G$ will usually be understood from the context, we will often ignore the dependence of $G$ and write $g(x),h(x),f(x;c)$ and $f(c)$ instead. It turns out that this last quantity is the one which determines the threshold for diameter~2 in random Cayley graphs. 

\begin{theorem}\label{General}
Let $c$ be a positive real number and let $(G_k)$ be a family of groups whose orders tend to infinity. Then $c$ is the threshold for diameter~2 for this family if and only if for every $\varepsilon > 0$ we have that $f_{G_k}(c-\varepsilon) \to \infty$ and $f_{G_k}(c+\varepsilon) \to 0$ as $k \to \infty$.
\end{theorem} 

This is a direct consequence of the following two results.

\begin{theorem}\label{General-1}
Let $c$ be a positive real number and let $(G_k)$ be a family of groups whose orders tend to infinity. Suppose that for every $\varepsilon > 0$ we have that $f_{G_k}(c-\varepsilon) \to \infty$ as $k \to \infty$. Then for every $\varepsilon > 0$, with high probability the diameter of $\Gamma \in \mathcal{G}\left(G_k,\sqrt{\frac{(c-\varepsilon)\log{|G_k|}}{|G_k|}}\right)$, is greater than~2.
\end{theorem}

\begin{theorem}\label{General-2}
Let $c$ be a positive real number and let $(G_k)$ be a family of groups whose orders tend to infinity. Suppose that for every $\varepsilon > 0$ we have that $f_{G_k}(c+\varepsilon) \to 0$ as $k \to \infty$. Then for every $\varepsilon > 0$, with high probability the diameter of $\Gamma \in \mathcal{G}\left(G_k,\sqrt{\frac{(c+\varepsilon)\log{|G_k|}}{|G_k|}}\right)$, is at most~2.
\end{theorem}

We proceed in the next two sections to prove these two theorems. The main ideas of the two proofs already appear in~\cite{CM2} and more specifically in the proofs of Theorems~1.3 and ~1.6 respectively.  

\section{Proof of Theorem~\ref{General-1}}

In the proof we will use the following result known as Kleitman's Inequality.

\begin{theorem}[Kleitman's Inequality]
Let $\Omega$ be a finite set and let $\{F_i\}_{i \in I}$ be subsets of $\Omega$, where $I$ is a finite index set. Let $R$ be a random subset of $\Omega$ and for each $i \in I$ let $E_i$ be the event that $F_i \subseteq R$. Then 
\[
\Pr\left(\bigcap_{i\in I} \overline{E_i} \right) \geqslant \prod_{i \in I} \Pr(\overline{E_i}).
\]
\end{theorem}

We refer the reader to~\cite{AS} for a proof of the inequality and continue with the proof of Theorem~\ref{General-1}.

Let $G = G_k$ where $k$ is large enough so that all of the following calculations hold. We  write $n$ for the order of $G$.

For any element $x$ of $G$, let us write $B_x$ for the event that $x$ has distance greater than~2 from the identity. Suppose first that $x^2 \neq 1$. Then we have that
\[
\Pr(B_x) \geqslant (1-p)^2(1-p)^{2\ell(\Gamma_x)}(1-p^2)^{e(\Gamma_x)}
\]
This follows directly from Kleitman's Inequality since 
\[B_x = \overline{E_x} \cap \overline{E_{x^{-1}}} \bigcap_{\{y:y^2=x\}} (\overline{E_y} \cap \overline{E_{y^{-1}}}) \bigcap_{e \in E(\Gamma_x)} \overline{E_e} ,\] 
where for an element $z$ of $G$, $E_z$ denotes the event that $z$ appears in the generating set and for an edge $e$ of $\Gamma_x$, $E_e$ denotes the event that both of its incident vertices appear in the generating set. Thus, provided $k$ and therefore $n$ is large enough, we have
\begin{align*}
\Pr(B_x) &\geqslant \frac{1}{2}\exp\left\{-\frac{2p\ell(\Gamma_x)}{1-p}\right\}\exp\left\{-\frac{p^2e(\Gamma_x)}{1-p^2}\right\} \\
&= \frac{1}{2}h(x)^{\frac{\sqrt{c-\varepsilon}}{1-p}}g(x)^{\frac{c-\varepsilon}{1-p^2}} \\ 
&\geqslant \frac{1}{2}h(x)^{\sqrt{c-\varepsilon/2}}g(x)^{c-\varepsilon/2} = \frac{1}{2}f(x;c-\varepsilon/2).
\end{align*} 
Here we have made use of the inequality $(1-x) \geqslant e^{-\frac{x}{1-x}}$ which holds for all $x \in [0,1)$.

If $x^2 = 1$, then for each $y$ with $y^2 = x$ we have $y^{-2} = x^{-1} = x$ as well, so
\[ \bigcap_{\{y:y^2=x\}} (\overline{E_y} \cap \overline{E_{y^{-1}}}) = \bigcap_{\{y:y^2=x\}} \overline{E_y}.\]
So in this case we again have
\begin{align*}
\Pr(B_x) &\geqslant \frac{1}{2}\exp\left\{-\frac{p\ell(\Gamma_x)}{1-p}\right\}\exp\left\{-\frac{p^2e(\Gamma_x)}{1-p^2}\right\} \\
&= \frac{1}{2}h(x)^{\frac{\sqrt{c-\varepsilon}}{1-p}}g(x)^{\frac{c-\varepsilon}{1-p^2}} \\ 
&\geqslant \frac{1}{2}h(x)^{\sqrt{c-\varepsilon/2}}g(x)^{c-\varepsilon/2} = \frac{1}{2}f(x;c-\varepsilon/2)
\end{align*} 

Let us now write $X$ for the number of vertices of $G$ which are at distance greater than~2 from the identity. Then
\[ \mathbb{E}X = \sum_{x \neq 1} \Pr(B_x) \geqslant \frac{1}{2} f_G(c-\varepsilon/2) \to \infty.\]
However this is not enough to show that $X > 0$ with high probability. The reason is that the some pairs of these events are highly correlated. For this reason, in general we will have $\Var(X) \gg (\mathbb{E}X)^2$ and so the second moment method cannot be used directly. To overcome this difficulty, we will give a quantitative result of the fact that if $B_x$ and $B_y$ are highly correlated then the graphs $\Gamma_x$ and $\Gamma_y$ have many common edges. Having proved this result, the next task will be to cleverly choose a subset $I$ of the elements of $G$ such that the events $(B_x)_{x \in I}$ are `almost' independent and furthermore it is still the case that $\sum_{x \in I} \Pr(B_x) \to \infty$.

\begin{lemma}\label{almost-independence}
Let $x,y$ be distinct elements of $G$ such that the graphs $\Gamma_x$ and $\Gamma_y$ have exactly $r$ common edges. Suppose furthermore that either $\ell(\Gamma_x) = o(\frac{n}{\log{n}})$ or $\ell(\Gamma_y) = o(\frac{n}{\log{n}})$. Then
\[ \Pr(B_x \cap B_y) \leqslant (1+o(1))\Pr(B_x) \Pr(B_y) (1-p^2)^{-r}.\] 
\end{lemma}  

\begin{proof}
We may assume that $\ell(\Gamma_y) = o(\frac{n}{\log{n}})$. Let us write $e_1,\ldots,e_r$ for the common edges of $\Gamma_x$ and $\Gamma_y$ and let $e_{r+1},\ldots,e_s$ be all other edges of $\Gamma_y$. By (O3) at most two of $e_1,\ldots,e_r$ are loops. For each $1 \leqslant i \leqslant s$, let $C_i$ be the event that not both vertices incident to $e_i$ appear in the generating set and let $D_i = \cap_{j=1}^i C_j$. Since $D_s = B_y$ then
\[
\Pr(B_x \cap B_y) = \Pr(B_x) \Pr(B_y) \frac{\Pr(B_x|B_y)}{\Pr(B_x)} = \Pr(B_x) \Pr(B_y) \prod_{i=1}^s \frac{\Pr(B_x|D_i)}{\Pr(B_x|D_{i-1})}
\]
where by convention $\Pr(B_x|D_0) := \Pr(B_x)$. 

By the law of total probability, for each $1 \leqslant i \leqslant s$ we have
\begin{align*}
\Pr(B_x|D_{i-1}) &= \frac{\Pr(B_x \cap D_{i-1})}{\Pr(D_{i-1})} = \frac{\Pr(B_x \cap D_{i-1} \cap C_i) + \Pr(B_x \cap D_{i-1} \cap \overline{C_i})}{\Pr(D_{i-1})} \\
&= \frac{\Pr(B_x | D_{i-1} \cap C_i)\Pr(D_{i-1} \cap C_i)}{\Pr(D_{i-1})} + \frac{\Pr(B_x \cap D_{i-1} \cap \overline{C_i})}{\Pr(D_{i-1})} \\
&= \Pr(B_x|D_i) \Pr(C_i|D_{i-1}) + \frac{\Pr(B_x \cap D_{i-1} \cap \overline{C_i})}{\Pr(D_{i-1})}.
\end{align*}

To bound the first term, we use Kleitman's Inequality to the events $\overline{C_i}$ and $\overline{D_{i-1}}$ to get 
\begin{align*}
\Pr(B_x|D_i) \Pr(C_i|D_{i-1}) &\geqslant \Pr(B_x|D_i)\Pr(C_i) \\
&= \begin{cases}
(1-p)\Pr(B_x|D_i) & \text{if $e_i$ is a loop,}\\
(1-p^2)\Pr(B_x|D_i) & \text{otherwise.}
\end{cases}
\end{align*}
To bound the second term, observe first that if $1 \leqslant i \leqslant r$ then $B_x \subseteq C_i$ and so $\Pr(B_x \cap D_{i-1} \cap \overline{C_i}) = 0$. If $r+1 \leqslant i \leqslant s$, let $X_i$ denote the event that no vertex adjacent to a vertex of $e_i$ in $\Gamma_x$ appears in the generating set and $Y_i$ the event that for each $j < i$ if $e_j$ meets $e_i$ then the vertex incident to $e_j$ but not $e_i$ is not in the generating set. Then, $B_x \cap D_{i-1} \cap \overline{C_i} = B_x \cap D_{i-1} \cap \overline{C_i} \cap X_i \cap Y_i$. Observe now that whether the two vertices incident to the edge $e_i$ appear in the generating set or not does not affect the outcome of the event $B_x \cap D_{i-1} \cap X_i \cap Y_i$, i.e. the events $B_x \cap D_{i-1} \cap X_i \cap Y_i$ and $\overline{C_i}$ are independent. Applying now Kleitman's Inequality to the events $\overline{B_x} \cup \overline{D_{i-1}},\overline{X_i}$ and $\overline{Y_i}$ we get
\begin{align*}
\Pr(B_x \cap D_{i-1} \cap \overline{C_i}) &= \Pr(\overline{C_i})\Pr(B_x \cap D_{i-1} \cap X_i \cap Y_i) \\
&\geqslant \Pr(\overline{C_i})\Pr(X_i)\Pr(Y_i)\Pr(B_x \cap D_{i-1})\\ &\geqslant
\begin{cases}
p(1-p)^{28}\Pr(B_x|D_{i-i})\Pr(D_{i-1}) & \text{if $e_i$ is a loop,}\\
p^2(1-p)^{28}\Pr(B_x|D_{i-i})\Pr(D_{i-1}) & \text{otherwise.}
\end{cases}
\end{align*}
To see the last inequality, recall that by (O1) $\Gamma_x$ and $\Gamma_y$ have maximum degree at most $8$ and so each of the events $X_i$ and $Y_i$ says that at most 14 elements of $G$ (at most seven for each vertex incident to $e_i$) do not appear in the generating set.

It follows that if $1 \leqslant i \leqslant r$ then $\Pr(B_x|D_i) \leqslant \Pr(B_x|D_{i-1})/(1-p^2)$ unless $e_i$ is a loop in which case we have $\Pr(B_x|D_i) \leqslant \Pr(B_x|D_{i-1})/(1-p)$. Also if $r+1 \leqslant i \leqslant s$ then
\[
\frac{\Pr(B_x|D_i)}{\Pr(B_x|D_{i-1})} \leqslant \frac{1 - p^2(1-p)^{28}}{1 - p^2} \leqslant 1 + 29p^3
\]
provided that $n$ is large enough unless $e_i$ is a loop in which case we have
\[
\frac{\Pr(B_x|D_i)}{\Pr(B_x|D_{i-1})} \leqslant \frac{1 - p(1-p)^{28}}{1 - p} \leqslant 1 + 29p^2
\] 
provided $n$ is large enough. So recalling that at most two edges $e_i$ with $1 \leqslant i \leqslant r$ are loops and at most $o(\frac{n}{\log{n}})$ edges are loops in general we get
\begin{align*}
\Pr(B_x \cap B_y) &\leqslant \Pr(B_x) \Pr(B_y)(1-p)^{-2}(1-p^2)^{-r}(1+29p^2)^{\ell(\Gamma_y)}(1 + 29p^3)^{4n} \\
&= (1 + o(1))\Pr(B_x) \Pr(B_y)(1-p^2)^{-r}
\end{align*}
as $p^2\ell(\Gamma_y) = o(1)$ and $p^3n = o(1)$.
\end{proof}

The next task is to choose a `large' set $I$ such that the events $B_x$ for $x \in I$ are `almost independent' in the sense of Lemma~\ref{almost-independence}. To this end we define a new graph $H$ with vertex set the elements of $G$ excluding the identity in which $x$ and $y$ are adjacent if and only if $\Gamma_x$ and $\Gamma_y$ have at least $n^{1 - \delta}$ common edges. (Later, we will set $\delta = \varepsilon/8$.) Since by (O1) $\Gamma_x$ has at most $4n$ edges and by (O4) every edge of $\Gamma_x$ belongs to at most $7$ other $\Gamma_y$'s, the degree of $x$ in $H$ is at most $28n^{\delta}$. Let us now write $x_1,\ldots,x_{n-1}$ for the vertices of $H$. We may assume that
\[ f(x_1;c-\varepsilon/2) \geqslant \cdots \geqslant f(x_{n-1};c-\varepsilon/2).\]
Now going through this order, we greedily pick an independent set $I$ of $H$ and let $Y$ be  the number of vertices $y \in I$ which are at distance greater than $2$ from the identity and furthermore they satisfy $\ell(\Gamma_y) = o(\frac{n}{\log{n}})$. Using (O3) we see that the number of $y$'s which do not satisfy $\ell(\Gamma_y) = o(\frac{n}{\log{n}})$ is $O(\log{n})$. For each such $y$, using (O2) we see that $\Gamma_y$ contains a matching of size at least $n/8$. Indeed if $M$ is a maximal matching of size $m$, then by (O2) every vertex outside $M$ has at least one neighbour inside $M$ and so there are at least $n-m$ edges with exactly one incident vertex inside $M$. On the other hand by (O1) every vertex inside $M$ has at most seven neighbours outside $M$ so there are at most $7m$ edges with exactly one incident vertex inside $M$. So $n-m \leqslant 7m$ which gives the claimed bound. Therefore, $P(B_y) \leqslant (1-p^2)^{n/8}$. In particular the sum over all such $y$ is $o(1)$. 

So since the maximum degree of $H$ is at most $28n^{\delta}$ this way of picking the elements of $I$ guarantees that
\[ \mathbb{E}Y \geqslant \frac{1}{2(28n^{\delta}+1)} f_G(c - \varepsilon/2) - o(1) \geqslant \frac{n^{-\delta}}{60}f_G(c - \varepsilon/2) - o(1) \]
But for each $x$ we have
\[
n^{-\delta}g(x)^{c - \varepsilon/2} = n^{-\delta - (c-\varepsilon/2)\frac{e(\Gamma_x)}{n}} \geqslant n^{-(c + 2\delta-\varepsilon/2)\frac{e(\Gamma_x)}{n}}
\]
since by (O2) we have $e(\Gamma_x) \geqslant n/2$. Thus, by taking $\delta = \varepsilon/8$ we obtain
\[ n^{-\delta}g(x)^{c - \varepsilon/2} \geqslant g(x)^{c - \varepsilon/4}\]
and so
\[ n^{-\delta}f(x,c - \varepsilon/2) \geqslant f(x,c - \varepsilon/4).\]
Therefore we get that
\[ \mathbb{E}Y \geqslant \frac{1}{60}f_G(c - \varepsilon/4) \to \infty. \]
Since also $\E Y(Y-1) \leqslant (1 + o(1))(\E Y)^2$ by Chebyshev's inequality we get
\[
\Pr(Y = 0) \leqslant \frac{\Var(Y)}{(\E Y)^2} = \frac{\E Y(Y-1)}{(\E Y)^2} - 1 + \frac{1}{\E Y} = o(1),
\]
from which we deduce that with high probability there are vertices of $\Gamma$ which are at distance greater than~2 from the identity, thus completing the proof of Theorem~\ref{General-1}.

\section{Proof of Theorem~\ref{General-2}}

Let $G$ be a group from the family $(G_k)$ and let us write $1$ for its identity element and $n$ for its order. Since the graph $\Gamma$ is vertex-transitive it is enough to show that with high probability every element of $G$ is at distance at most~2 from the identity. Now fix a non-identity element $x$ of $G$ and for each $y \neq 1,x$ let us denote by $A_y:=A_y(1,x)$ the event that the edges between $1$ and $y$ and between $x$ and $y$ both appear in $\Gamma$. So we have
\[ \Pr(d(1,x) > 2) \leqslant \Pr \left( \bigcap_{y \neq 1,x} \overline{A_y} \right).\]
The events $A_y$ are not independent and so to estimate the right hand side of the above inequality we will make use of the following inequality of Janson. We refer the reader to~\cite{AS} for its proof. 

\begin{theorem}[Janson's Inequality]
Let $\Omega$ be a finite set and let $\{F_i\}_{i \in I}$ be subsets of $\Omega$, where $I$ is a finite index set. Let $R$ be a random subset of $\Omega$ and for each $i \in I$ let $E_i$ be the event that $F_i \subseteq R$. Suppose also that $\Pr(E_i) \leqslant \eps$ for each $i \in I$. Then
\[
\Pr\left(\bigcap_{i\in I} \overline{E_i} \right) \leqslant \exp\left(-\sum_{i \in I} \Pr(E_i) + \frac{1}{1-\varepsilon}\sum_{i \in I}\sum_{\{j \neq i : F_i \cap F_j \neq \emptyset\}} \Pr(E_i \cap E_j) \right).
\]
\end{theorem}

We will not apply this result directly to all of the events $A_y$ but only to a specially chosen subset of them.

We start by defining an equivalence relation on $G$ by letting $y$ be equivalent to $z$ if and only if $y=z$ or $y=z^{-1}$. We write $[y]$ for the equivalence class of $y$ and we let $\Omega$ be the set of all equivalence classes. We let $R$ be a random subset of $\Omega$ where each equivalence class is chosen independently with probability $p$ or $2p-p^2$ depending on whether the equivalence class contains one or two elements. For each element $i$ of $G$ we let $F_i = F_i(x) = \{[i],[xi^{-1}]\}$ and we let $E_i$ be the event that $F_i \subseteq R$. For any $I \subseteq V(G) \setminus \{1,x\}$ we have
\[ \Pr(d(1,x) > 2) \leqslant \Pr \left( \bigcap_{y \neq 1,x} \overline{A_y} \right) \leqslant \Pr \left( \bigcap_{i \in I} \overline{E_i} \right).\]

Now let $I = J \cup I_0 \cup I_1 \cup I_2$ be any subset of $V(G) \setminus \{1,x\}$ with the following properties.
\begin{itemize}
\item[(P1)] For each $i \in J$ we have $|F_i| = 1$.
\item[(P2)] For each $i \in I_0 \cup I_1 \cup I_2$ we have $|F_i| = 2$. Furthermore exactly $t$ of the equivalence classes in $F_i$ have size $2$ if and only if $i \in I_t$.
\item[(P3)] For each $j \in J$ and $i \in I_0 \cup I_1 \cup I_2$ we have that $F_i \cap F_j = \emptyset$.
\item[(P4)] For each distinct $i,j \in I_0 \cup I_1 \cup I_2$ we have $F_i \neq F_j$.
\end{itemize}

Here the sets $I,J,I_0,I_1,I_2$ depend on $x$ and later we will have such sets for each $x$. In order to distinguish between them, we will sometimes write $I(x),J(x),I_0(x),I_1(x)$ and $I_2(x)$ instead. However if we feel that there is no danger of confusion we will choose to drop the dependence on $x$.

Suppose now that we are given such sets satisfying properties (P1)-(P4). Observe that for each $i \in J$ we have $i^2 = x$, so $i \neq i^{-1}$ and therefore $\Pr(E_i) = 2p-p^2$. Observe also that for each $i \in I_0$ we have $\Pr(E_i) = p^2$, for each $i \in I_1$ we have $\Pr(E_i) = p(2p-p^2)$ and for each $i \in I_2$ we have $\Pr(E_i) = (2p-p^2)^2$. In particular for each $t \in \{0,1,2\}$ and each $i \in I_t$ we have $\Pr(E_i) = 2^tp^2 + O(p^3)$. Finally observe that for each $i$ there are at most $6$ other $j$'s for which $F_i \cap F_j \neq \emptyset$. Indeed $F_i \cap F_j \neq \emptyset$ if and only if $[j] = [i]$ or $[j] = [xi^{-1}]$ or $[xj^{-1}] = [i]$ or $[xj^{-1}] = [xi^{-1}]$ which (for $j \neq i$) happens if and only if $j \in \{i^{-1},xi^{-1},ix^{-1},ix,ix^{-1},xi^{-1}x\}$.

So applying Janson's inequality with this set $I$ we get
\[ \Pr(d(1,x) > 2) \leqslant \exp\left(-2p|J| + p^2|J| - p^2|I_0| - 2p^2|I_1| - 4p^2|I_2| + O(np^3) \right).\]

We continue by choosing the sets $J(x)$ that we are going to use. To this end let $K(x) = \{i \in G:i^2=x\}$ and observe that for (P1) to hold we must have $J(x) \subseteq K(x)$.  If $x^2 \neq 1$ we set $J(x) = K(x)$, while if $x^2 = 1$ then we choose $J(x)$ such that it contains exactly one element from each pair $\{i,i^{-1}\}$ contained in $K$. Note that all these pairs really have size two since if it was the case that $i = i^{-1}$ then we would have $i^2 = 1 \neq x$ and so $i \notin K(x)$. So with this definition, in the first case we have $|J(x)| = \ell(\Gamma_x)$ while in the second case we have $|J(x)| = \ell(\Gamma_x)/2$. In particular, in both cases we get
\[
\Pr(d(1,x) > 2) \leqslant (1+o(1))h(x)^{\sqrt{c+\varepsilon}}n^{-(c+\varepsilon)(|I_0|+2|I_1| +4|I_2| - |J|)/n}.
\]
Suppose now that for each $x$ we could find $I_0(x),I_1(x),I_2(x)$ such that properties (P2)-(P4) are satisfied and furthermore we have that
\begin{equation}\label{EQ1}
|I_0(x)| + 2|I_1(x)| + 4|I_2(x)| - |J(x)| \geqslant (1+o(1))e(\Gamma_x).
\end{equation}
Then we would get 
\[
\Pr(d(1,x) > 2) \leqslant (1+o(1))f(x;c+\varepsilon/2)
\]  
and so the probability that there is an $x$ with distance greater than~2 from the identity would be at most 
\[
\sum_x \Pr(d(1,x) > 2) \leqslant (1+o(1))f(c + \varepsilon/2) \to 0
\]  
and therefore by the first moment method with high probability the diameter of $\Gamma$ would be at most~2. 

In fact, because for each $x$ for which $|J(x)| \geqslant n^{2/3}$ we have that
\[ 
\Pr(d(1,x) > 2) \leqslant (1+o(1))\exp\{-(2p-p^2)|J(x)|\} \leqslant (1+o(1))e^{-p|J(x)|} = o(1/n)
\]
it is enough to find sets $I_0(x),I_1(x),I_2(x)$ satisfying properties (P2)-(P4) and~\eqref{EQ1} only for those $x$ for which $|J(x)| < n^{2/3}$. Furthermore in this case we don't even need to satisfy property $(P3)$. Indeed as we observed earlier, for each $j \in J(x)$ there are at most six other $i$'s with $F_i \cap F_j \neq \emptyset$. So having found sets $I_0(x),I_1(x),I_2(x)$ satisfying properties (P2),(P4) and~\eqref{EQ1} by removing at most $6|J| = o(n)$ elements from each of them we will get sets $I_0'(x),I_1'(x),I_2'(x)$ satisfying (P2)-(P4) and since $e(\Gamma_x) \geqslant n/2$, the inequality~\eqref{EQ1}
will still be satisfied.

So pick an element $x$ of $G$ which is different from the identity and assume that $|J(x)| < n^{2/3}$. Let us pick an element $y$ of $G$ distinct from the identity for which $y^2 \neq x$ and $y^{-2} \neq x$. The condition $|J(x)| < n^{2/3}$ guarantees that all but $o(n)$ elements $y$ satisfy these properties. Our first task is to understand how many neighbours does $y$ have in the graph $\Gamma_x$. Recall that $N(y) = \{xy,xy^{-1},x^{-1}y,x^{-1}y^{-1},yx,y^{-1}x,yx^{-1},y^{-1}x^{-1}\}$. So we want to find the size of this set and this of course depends on certain relations that $x$ and/or $y$ may satisfy. All the information required is conveyed in Table~\ref{T1}.

\begin{table}[h]
\begin{center}
\begin{tabular}{|c|c||c|c|c|c|c|c|}
\hline
R1 & $x^2=1$ & $1,3$ & $2,4$ & $5,7$ & $6,8$ & $B,E$ & $B=E$\\ \hline
R2 & $y^2=1$ & $1,2$ & $3,4$ & $5,6$ & $7,8$ & $\ast$ & $B=C$ \\ \hline
R3 & $xy=yx$ & $1,5$ & $2,6$ & $3,7$ & $4,8$ & $C,D$ & $C=D$\\ \hline
R4 & $(xy)^2=x^2$ & $1,6$ & $2,5$ & $3,8$ & $4,7$ & $B,D$ & $B=D$\\ \hline
R5 & $(xy)^2=y^2$ & $1,7$ & $2,8$ & $3,5$ & $4,6$ & $A,F$ & $A=F$\\ \hline
R6 & $(xy)^2=1$ & $1,8$ & $4,5$ &  & & & \\ \hline
R7 & $(xy^{-1})^2=1$ & $2,7$ & $3,6$  &  & & $\ast$ & $D=E$\\ \hline
R8 & $x^2=y^2$ & $2,3$ & $6,7$ &  & &  $C,E$ & $C=E$\\ \hline
R9 & $x^2=y^{-2}$ & $1,4$ & $5,8$ & & & & \\ \hline
\end{tabular}
\caption{}
\label{T1}
\end{center}
\end{table}  

We can use Table~\ref{T1} to interpret in which instances the possible neighbours of $y$ are actually equal. Here we list the possible neighbours in the order $xy$, $xy^{-1}$, $x^{-1}y$, $x^{-1}y^{-1}$, $yx$, $y^{-1}x$, $yx^{-1}$, $y^{-1}x^{-1}$. For example, the first row of the table says that whenever $x^2=1$ then the first expression $xy$ is equal to the third expression $x^{-1}y$, the second is equal to the fourth and so on. The last two columns of Table~\ref{T1} will be explained later. We can now use Table~\ref{T1} to construct Table~\ref{T2}. 

\begin{table}[h]
\begin{center}
\begin{tabular}{|c|c||c|c|c|}
\hline
1 & R1,R2,R3 & $12345678$ & $\ast \ast$ & $BCDE$\\ \hline
2 & R1,R2 & $1234|5678$ & $\ast$ & $BCE$\\ \hline
3 & R1,R3 & $1357|2468$ & & $AF|BE|CD$ \\ \hline
4 & R1,R4 & $1368|2457$ & $\ast$ & $BDE$ \\ \hline
5 & R1 & $13|24|56|78$ & & $BE$ \\ \hline
6 & R2,R3 & $1256|3478$ & $\ast$ & $BCD$\\ \hline
7 & R2,R5 & $1278|3456$ & $\ast \ast$ & \\ \hline
8 & R2 & $12|34|56|78$  & $\ast$ &\\ \hline
9 & R3,R6,R7 & $1458|2367$  & $\ast$ & $CDE$ \\ \hline
10 & R3,R6 & $1458|26|37$  & & $CD$ \\ \hline
11 & R3,R7 & $15|48|2367$  & $\ast$& $CDE$ \\ \hline
12 & R3 & $15|26|37|48$  & & $CD$ \\ \hline
13 & R4,R5 & $1467|2358$ & & $AF|BD|CE$ \\ \hline 
14 & R4 & $16|25|38|47$  & & $BD$ \\ \hline
15 & R5 & $17|28|35|46$  & & $AF$ \\ \hline
16 & R6,R7 & $18|27|36|45$ & $\ast$ &\\ \hline 
17 & R6,R8 & $18|23|45|67$ & & $CE$ \\ \hline
18 & R6 & $18|45|2|3|6|7$  & &\\ \hline
19 & R7,R9 & $14|27|36|58$ & $\ast$ & \\ \hline 
20 & R7 & $1|4|5|8|27|36$  & $\ast$ &\\ \hline
21 & R8,R9& $14|23|58|67$  & & $CE$ \\ \hline
22 & R8& $1|4|5|8|23|67$   & & $CE$ \\ \hline
23 & R9& $14|58|2|3|6|7$ & & \\ \hline
24 & & $1|2|3|4|5|6|7|8$  & &\\ \hline
\end{tabular}
\caption{}
\label{T2}
\end{center}
\end{table} 
  
From Table~\ref{T2} we can read how many neighbours each $y$ has depending on which relations $x$ and $y$ satisfy. For example the third row says that if $x$ and $y$ satisfy the relations in rows 1 and 3 of Table~\ref{T1} and all the relations that can be deduced by them but no relation from any other row, then $y$ has exactly two neighbours. This is because the first, third, fifth and seventh expression in the above list are all equal and also the second, fourth, sixth and eight relation are again all equal but distinct from the others (as no other relation holds). The last two columns of Table~2 will be discussed later. We now explain how the second of Table~\ref{T2} is constructed. (The construction of the third column from the second is obvious.) In principle, Table~{T2} should have $2^9$ rows, one for each subset of relations that appear in Table~\ref{T1}. However since some of these relations together imply some others we can cut this down a lot. We start by setting all relations one by one to be true until they force all other relations to be true. For example, once the relations in rows R1-R3 of Table~\ref{T1} are true, then so are all the other. This creates the (second column of the) first row of Table~\ref{T2}. Having constructed the $k$-th row of Table~\ref{T2} we construct its $(k+1)$-th row as follows: If  in the second column of the $k$-th row we have the elements $Ri_1,\cdots,Ri_s$ with $i_1 < \cdots < i_s$, then for the $(k+1)$-th row we set the relations $Ri_1,\ldots,Ri_{s-1}$ as true, all other relations $Ri_t$ with $t \leqslant s$ as false and then we continue by setting all other relations one by one as true until these force the rest to be either true or false.  

Recall that we do the above calculations only for elements $y$ for which $x \neq y^2,y^{-2}$. So when computing $E(\Gamma_x)$ we will miss all neighbours of such $y$'s. However we already observed that there are at most $2|J(x)| < 2n^{2/3} = o(n)$ such $y$ so the total contribution that we miss is only $o(E(\Gamma_x))$.

Our second task now is for each $y$ with $y \neq 1,x,x^{-1}$ and $y^2,y^{-2} \neq x$ to understand the sizes of the equivalence classes of $F_y$ and also how many distinct $z$'s are there for which $F_y \neq F_z$. We begin by recalling that $F_y \cap F_z$ is not empty only in the cases where $z=y^{-1},yx,y^{-1}x,xy^{-1},yx^{-1},xy^{-1}x$ in which cases we have $F_{y^{-1}} = \{[y],[xy]\}$, $F_{yx} = \{[y],[yx]\}$, $F_{y^{-1}x} = \{[y],[y^{-1}x]\}$, $F_{xy^{-1}} = \{[xyx^{-1}],[xy^{-1}]\}$, $F_{yx^{-1}} = \{[x^2y^{-1}],[xy^{-1}]\}$ and $F_{xy^{-1}x} = \{[xy^{-1}x],[xy^{-1}]\}$. Let us write $A = y^{-1}, B=yx,\ldots,F=xy^{-1}x$. Now let us go back to Table~1. The last column denotes equalities between $A,B,\ldots,F$. For example the first row says that whenever $x^2=1$ then $B=E$ (i.e.~$yx = yx^{-1}$) and no other relation must necessarily hold. The second to last column says for which $z$ we actually have $F_y = F_z$. For example again from the first row we read that if $x^2 = 1$ then we must have $F_y = F_B = F_E$ but no other equality must necessarily hold. The stars in the second and seventh row denote the fact that some of the equivalence classes of $F_y$ will have size one instead of two. For example if $y^2=1$ then the class $[y]$ of $F_y$ has size one. 

We now discuss the last two columns of Table~2. The second to last column denotes how many equivalence classes of $F_y$ have size one. The number of them is the number of stars appearing in this column. The last column denotes which $F_z$'s are neighbours to $F_y$. For example in the first row we have that all $F_B,F_C,F_D,F_E$ are equal to $F_y$ and moreover we actually have $B=C=D=E$. While in the third row we see that all $F_A,\ldots,F_F$ are equal to $F_y$ but in fact the only equalities between $A,B,\ldots,F$ are $A=F,B=E$ and $C=D$.

Let us now define a graph $R(x)$ on the set of all $y$ with $y\neq 1,x,x^{-1}$ and $x\neq y^{2},y^{-2}$ in which we join $y$ to $z$ if and only if $F_y = F_z$. We see from Table~2 that this graph is a union of disjoint $K_2$'s and $K_4$'s. To satisfy conditions (P2) and (P4) it is necessary and sufficient the the elements we pick to form $I_0(x) \cup I_1(x) \cup I_2(x)$ are an independent set of $R(x)$. We now go through each $K_2$ and $K_4$ and pick one element at random. We calculate $\mathbb{E}(|I_0(x)| + 2|I_1(x)| + 4|I_2(x)|)$. The first row of Table~\ref{T2} says that each $y$ of this type contributes $\frac{1}{2} \cdot 1 = \frac{1}{2}$ to this expectation. Here the $1/2$ comes from the fact that this element is chosen with probability $1/2$, while the $1$ comes from the fact that this element will belong to $I_0(x)$. Because every such $y$ has exactly one neighbours, it also contributes $1/2$ to the count of $E(\Gamma_x) = \sum_y |N_{\Gamma_x}(y)|/2$. We now go through each row and we discover that the contribution to $\mathbb{E}(|I_0(x)| + 2|I_1(x)| + 4|I_2(x)|)$ in each one of them is equal to the contribution to $E(\Gamma_x)$ unless in the following instances: 
\begin{itemize}
\item In row 10 we get a contribution of $2$ instead of $3/2$.
\item In row 11 we get a contribution of $1$ instead of $3/2$.
\item In row 18 we get a contribution of $4$ instead of $3$.
\item In row 20 we get a contribution of $2$ instead of $3$.
\item In row 22 we get a contribution of $2$ instead of $3$.
\item In row 23 we get a contribution of $4$ instead of $3$.
\end{itemize}
However, the number of $y$'s which satisfy the relations of row 10 is equal to the number of $y$'s which satisfy the relations of row 11. Indeed one can check that $y$ satisfies the relations of row 10 if and only if $y^{-1}$ satisfies the relations of row 11. Similarly, again under the transformation $y \mapsto y^{-1}$, we see that the number of $y$'s which satisfy the relations of row 18 is the same as those which satisfy the relations of row 20 and the same happens with rows 22 and 23. So we have
\[ \mathbb{E}(|I_0(x)| + 2|I_1(x)| + 4|I_2(x)|) \geqslant (1+o(1))E(\Gamma_x)\]
where the $1+o(1)$ factor comes from the fact that we did not consider $y$'s with $x=y^2$ or $x=y^{-2}$. In particular there is a choice of $y$'s for which
\[ |I_0(x)| + 2|I_1(x)| + 4|I_2(x)| \geqslant (1+o(1))E(\Gamma_x).\]
This completes the proof of Theorem~\ref{General-2}.

\section{Proofs of Theorems~\ref{SmallCase}-\ref{InterestingCase}}

We begin with the proof of Theorem~\ref{LargeCase} which is the simplest of the four. To avoid unnecessary repetitions both in this and the proofs that follow, whenever we refer to the neighbours of a vertex in a graph we will mean the neighbours of this vertex in the underlying simple graph.

\begin{proof}[Proof of Theorem~\ref{LargeCase}] 
Let $G = (C_2)^{n} \times (C_m)$ where $m = \lfloor 2^{\frac{(1-c)n}{c}}\rfloor$ and let us write $N$ for the order of $G$. So $2^n \sim N^c$ and $m \sim N^{1-c}$. We now compute $f(x;c')$ for each $x \in G$.
\begin{itemize}
\item[(a)] If $x$ is of the form $(x',1)$ where $1$ is the identity in $C_m$, then $x^2 = 1$, $\Gamma_x$ has no loops, and each $y$ has exactly two neighbours in $\Gamma_x$, namely $xy$ and $xy^{-1}$ unless $y^2=1$. Since there are at most $2^{n+1} = o(N)$ such $y$'s we have $e(\Gamma_x) = (1+o(1))N$ and so $f(x,c') = N^{-c'+o(1)}$. Since there are $2^n \sim N^c$ such elements $x$ the total contribution of these elements to $f(c')$ is $N^{c-c'+o(1)}$ which tends to infinity if $c' < c$ and tends to 0 if $c' > c$.
\item[(b)] If $x$ is not as in (a) but still satisfies $x^2=1$, then each $y$ has exactly two neighbours in $\Gamma_x$, namely $xy$ and $xy^{-1}$ unless $y^2 \in \{1,x\}$. Since there are at most $o(N)$ such $y$'s we have $e(\Gamma_x) = (1 + o(1))N$ and so $f(x,c') = N^{-c'+o(1)}$. Since there are at most $2^n \sim N^c$ such elements $x$ the total contribution of these elements to $f(c')$ is $N^{c-c'+o(1)}$ which tends to 0 if $c' > c$.
\item[(c)] If $x$ is not as in (a) or (b) then all but $o(N)$ $y$'s have exactly four neighbours in $\Gamma_x$, namely $xy,xy^{-1},x^{-1}y,x^{-1}y^{-1}$. Therefore $e(\Gamma_x) = (2+o(1))N$ and so $f(x,c') \leqslant N^{-2c'+o(1)}$. In particular the total contribution of such elements to $f(c')$ is at most $N^{1-2c' + o(1)}$ which (since $c \geqslant 1/2$) tends to 0 if $c' > c$.   
\end{itemize}
Combining the calculations in (a)-(c) we see that $f(c')$ tends to infinity if $c' < c$ and to $0$ if $c' > c$ as required. 
\end{proof}

We now proceed with the proof of Theorem~\ref{SmallCase}. In the proof we will repeatedly make use of the facts that for each element $x$ of $S_n$ there are at most $(n!)^{1/2 + o(1)}$ elements $y$ with $y^2 = x$ and (unless $x$ is the identity) at most $(n-2)!$ elements which commute with $x$. (The second assertion is a simple exercise. To prove the first assertion, it is enough by~\cite[Exercise 7.69.c]{Stanley} to prove it when $x$ is the identity. A proof of this is included in~\cite{CM2}.) In particular for each element $x$ of $S_n$ there are at most $o(n!)$ elements $y$ for which the set $\{xy,xy^{-1},x^{-1}y,x^{-1}y^{-1},yx,y^{-1}x,yx^{-1},y^{-1}x^{-1}\}$ has size less than~$8$. The only case where this might not be immediately obvious is that for each $x$ there are at most $o(n)$ elements $y$ with $(xy)^2 = y^2$. But if there is a $z$ with $(xz)^2 = z^2$ then $z^{-1}xz = x^{-1}$. So $x^{-1}$ belongs to the equivalence class of $x$ and so the number of $y$'s with $(xy)^2=y^2$ is the same as the number of $y$'s for which $y^{-1}xy = x$, i.e.~it is the same as the number of elements which commute with $x$.

\begin{proof}[Proof of Theorem~\ref{SmallCase}] 
Let $G = S_n \times C_m$ where $m = \lfloor n!^{\frac{2c}{1-2c}}\rfloor$ and let us write $N$ for the order of $G$. So $n! \sim N^{1-2c}$ and $m \sim N^{2c}$. 

We begin by computing $h(x)$ for each $x \in G$. Since for each $x \in C_m$ there are at most two $y$'s in $C_m$ with $y^2 = x$ and for each $x \in S_n$ there are at most $n!^{1/2 + o(1)}$ $y$'s in $S_n$ with $y^2 = x$ then for each element $x \in G$ there are at most $N^{1/2 - c + o(1)}$ elements $y$ with $y^2 = x$. In particular $\ell(\Gamma_x) = O(N^{1/2 - c + o(1)})$ and so $h(x) = (1+o(1))$.

We continue by computing $g(x)$ and thus $f(x;c')$ for each $x \in G$.
\begin{itemize}
\item[(a)] There is at most one (non-identity) element $x$ of the form $(1,x')$, where $1$ is the identity of $S_n$, with $x^2=1$. It contributes $o(1)$ to $f(c')$.
\item[(b)] Suppose $x$ is of the form $(1,x')$ where $1$ is the identity of $S_n$ and $x^2 \neq 1$. Then all but $o(N)$ elements $y$ have exactly four neighbours in $\Gamma_x$, namely $xy,xy^{-1},x^{-1}y$ and $x^{-1}y^{-1}$. Therefore $e(\Gamma_x) = (1+o(1))N$ and so $f(x,c')=N^{-2c' + o(1)}$. Since there are either $m$ or $m-1$ such elements, the total contribution of these elements to $f(c')$ is $N^{2c - 2c' + o(1)}$ which tends to infinity if $c' < c$ and tends to 0 if $c' > c$.
\item[(c)] If $x$ is not of the above forms but $x^2=1$, then all but $o(N)$ elements $y$ have exactly four neighbours in $\Gamma_x$, namely $xy,xy^{-1},yx$ and $y^{-1}x$. Therefore $e(\Gamma_x) = (2+o(1))N$ and so $f(x,c')=N^{-2c' + o(1)}$. Since there are (at most) $n!^{1/2 + o(1)} = N^{1/2 - c + o(1)}$ such elements the total contribution of these elements to $f(c')$ is at most $N^{1/2 - c - 2c'}$ which (since $c \geqslant 1/4$) tends to $0$ if $c' > c$.
\item[(d)] If $x$ is not of the above forms then all but $o(N)$ elements $y$ have exactly eight neighbours in $\Gamma_x$. Therefore $e(\Gamma_x) = (4+o(1))N$ and so $f(x,c')=N^{-4c' + o(1)}$. In particular the total contribution of these elements to $f(c')$ is at most $N^{1 - 4c' + o(1)}$ which (since $c \geqslant 1/4$) tends to $0$ if $c' > c$.
\end{itemize}
Combining the calculations in (a)-(d) we see that $f(c')$ tends to infinity if $c' < c$ and to $0$ if $c' > c$ as required. 
\end{proof}

\begin{proof}[Proof of Theorem~\ref{LargerCase}]
Let $G = (C_2)^{n} \times (D_{2m})$ where $m = \lfloor 2^{\frac{(4-3c)n}{3c}}\rfloor$ and $D_{2m}$ is the dihedral group of order $2m$ and let us write $N$ for the order of $G$. So
$2^n = \Theta(N^{3c/4})$ and $2m =  \Theta(N^{1-3c/4})$. We now compute $f(x;c')$ for each $x \in G$.
\begin{itemize}
\item[(a)] Suppose $x$ is of the form $(x',1)$ where $1$ is the identity element of $D_{2m}$ and $x'$ is a non-identity element of $C_2^n$. Since no element of $G$ squares to $x$ then $\Gamma_x$ has no loops and so $h(x) = 1$. Note that $x$ is central satisfying $x^2=1$ and therefore every $y$ has at most two neighbours in $\Gamma_x$, namely $xy$ and $xy^{-1}$. Moreover it has exactly one neighbour if and only if $y$ is an involution. Since $D_{2m}$ has either $m+1$ or $m+2$ involutions depending on the parity of $m$ we have that $G$ has $\Theta(N/2)$ involutions and so $e(\Gamma_x) = (3/4 + o(1))N$ and $f(x;c') = N^{-3/4c' + o(1)}$. Since there are $2^n = \Theta(N^{3c/4})$ such elements, the total contribution of these elements to $f(c')$ is $N^{3c/4 - 3c'/4 + o(1)}$ which tends to infinity if $c' < c$ and tends to 0 if $c' > c$.
\item[(b)] Suppose $x$ is of the form $(1,x')$ where $1$ is the identity in $(C_2)^n$ and $x'$ is a non-identity element of $D_{2m}$ representing a rotation of the regular $m$-gon. Then the equation $y^2=x$ has at least $2^n = \Theta(N^{3c/4})$ solutions. Since $c>1$, we deduce that $h(x) = e^{-\Omega(n^{1/4})}$ and so $f(x;c') = o(1/N)$. So the total contribution of these elements to $f(c')$ is $o(1)$.
\item[(c)] For every other element $x$ we have that $\Gamma_x$ contains no loops. Moreover, either we have that $x$ is not an involution, in which case every $y$ has at least two neighbours in $\Gamma_x$, namely $xy$ and $x^{-1}y$, or $x$ is an involution in which case all but at most $O(2^n) = O(N^{3/4})$ elements $y$ commute with $x$ with the rest having at least two neighbours in $\Gamma_x$, namely $xy$ and $yx$. So in both cases we have $e(\Gamma_x) \geqslant (1+o(1))N$, and therefore $f(x,c') = 1/N^{c'} = o(1/N)$. So the total contribution of these elements to $f(c')$ is also $o(1)$. 
\end{itemize}
Combining the calculations in (a)-(d) we see that $f(c')$ tends to infinity if $c' < c$ and to $0$ if $c' > c$ as required. 
\end{proof}

As we mentioned in the introduction the proof of Theorem~\ref{InterestingCase} is based on the fact that there are finitely many values that can appear as the proportion of involutions of a finite group in the interval $[1/2,1]$. The earliest reference we could locate for this result is~\cite{Miller}. Here, we will make use of the following theorem of Wall~\cite{Wall} which classifies all groups for which more than half of their elements are involutions.

\begin{theorem}\label{InvolutionTheorem}
Let $G$ be a group having exactly $(1/2 + c)|G|$ involutions, where $c > 0$. Then there is a positive integer $n$ such that $c = 1/2n$. Moreover, $G$ is the direct product of copies of $C_2$ together with a group $H$ which has one of the following forms:
\begin{itemize}
\item[(I)] $H$ contains an abelian subgroup $H_1$ of index 2 and an element $g$ with $gh=h^{-1}g$ for every $h \in H$. Moreover $C_2$ does not appear as a factor in the decomposition of $H_1$ as products of cyclic groups of prime power orders.
\item[(II)] $H = D_8 \times D_8$.
\item[(III)] $H$ is generated by involutions $c,x_1,y_1,\ldots,x_r,y_r$ for which all of them commute with each other apart from the pairs $\{x_i,y_i\}$ for $1 \leqslant i \leqslant r$ which satisfy $(x_iy_i)^2=c$.
\item[(IV)] $H$ is generated by involutions $c,x_1,y_1,\ldots,x_r,y_r$ for which all of them commute with each other apart from the pairs $\{c,x_i\}$ for $1 \leqslant i \leqslant r$ which satisfy $(cx_i)^2=y_i$.
\end{itemize} 
In particular, for each $d > 0$, there is a finite set $S(d)$ of groups such that every group $G$ with more than $(1/2 + d)|G|$ involutions is a direct product of copies of $C_2$ together with a group from $S(d)$.
\end{theorem}

Armed with this result we can move on to the proof of Theorem~\ref{InterestingCase}.

\begin{proof}[Proof of Theorem~\ref{InterestingCase}]
Suppose there is a family $(G_k)$ of groups for which $c > 4/3$ is a threshold for diameter 2. Let $c' = (c+4/3)/2$. By Theorem~\ref{General} we have that $f_{G_k}(c') \to \infty$. Let $G$ be a group from the family $(G_k)$ for which $f_G(c')$ is sufficiently large and let us write $N$ for its order and $\alpha N$ for the number of its involutions. We now compute $f(x;c')$ for each $x \in G$. We may assume that the equation $x = y^2$ has at most $N^{1/2}$ solutions as otherwise $\ell(\Gamma_x) \geqslant N^{1/2} $ which implies that $f(x,c') = o(1/N)$ and the total contribution of such elements to $f(c')$ is $o(1)$.
\begin{itemize}
\item[(a)] Suppose that $x$ is not an involution. Then each $y$ has at least two neighbours in $\Gamma_x$, namely $xy$ and $x^{-1}y$ and so $e(\Gamma_x) \geqslant N$. It follows that $f(x;c') \leqslant 1/N^{c'} = o(1/N)$ and so the total contribution of these elements to $f(c')$ is $o(1)$.
\item[(b)] Suppose that $x$ is an involution which is not central. Let $C(x) = \{y\in G:yx=xy\}$ be the centraliser of $x$ and suppose that it has size $\beta N$. Since $C(x)$ is a subgroup of $G$ and since $x$ is not central we must have $\beta \leqslant 1/2$. The only $y$'s which have exactly one neighbour in $\Gamma_x$ must belong to $C(x)$ so 
\[e(\Gamma_x) \geqslant (\beta + 2(1-\beta) + o(1))n/2 = (1 - \beta/2 + o(1))N \geqslant (3/4 + o(1))N.\]
It follows that $f(x;c') \leqslant N^{-3c'/4 + o(1)} = o(1/N)$. So again the total contribution of these elements to $f(c')$ is $o(1)$.
\item[(c)] If now $x$ is a central involution, then each $y$ which does not satisfy $y^2=x$ has either one or two neighbours in $\Gamma_x$ depending on whether $y$ is an involution or not. So $e(\Gamma_x) = (1 - \alpha/2 + o(1))N$. In particular, if $\alpha \leqslant 1/2 + d$, then the total contribution of these elements to $f(c')$ is at most $N^{1-c'(3/4 - d/2 + o(1))}$ which is $o(1)$ if $d > 0$ is sufficiently small.
\end{itemize}
Combining the calculations in (a)-(c) we see that unless the proportion of involutions in $G$ is at least $1/2+d$ for some sufficiently small but fixed $d$, then $f(c')$ tends to $0$, a contradiction.

So we can now use Theorem~\ref{InvolutionTheorem} to see that $G$ must be a direct product of copies of $C_2$ together with a group $H$ from a finite set $S(d)$ of groups. In particular, if $M(d)$ is the largest cardinality of a group in $S(d)$ then we have that $G$ contains at least $N/M(d)$ central involutions, namely all elements of the form $(x,1)$ where $x$ belongs to the direct product of the $C_2$'s and $1$ is the identity element of $H$. In particular, it contains at least $\Theta(N)$ central involutions $x$ for which the equation $y^2=x$ has at least $n^{1/3}$ solutions. For each such $x$ we have $g(x) = (1+o(1))$ and so by (c) the total contribution of these elements to $f(c')$ is $\Theta(N^{1 - c'(1 - \alpha/2)})$.  This tends to infinity if $c' < \frac{2}{2-\alpha}$ and tends to 0 if $c' > \frac{2}{2-\alpha}$. So $c = \frac{2}{2-\alpha}$ which is equal to $\frac{4n}{3n-1}$ if $\alpha = 1/2 + 1/2n$. To complete the proof of the theorem, we just observe that the given family has a proportion of $1/2 + 1/2n$ involutions as so does $D_{4n}$. 
\end{proof}

\section{Conclusion and open problems}

Even though we stated the results with fixed $\varepsilon$, the same proofs work to show the following

\begin{theorem}
Let $c$ be a positive real number and let $(G_k)$ be a family of groups of order $n_k$ with $n_k$ tending to infinity. Then $c$ is the threshold for diameter~2 for this family if and only if  $f_{G_k}(c_k) \to \infty$ whenever $c-c_k = \omega(\frac{\log{\log{n_k}}}{\log{n_k}})$ and $f_{G_k}(c_k) \to 0$ whenever $c_k-c = \omega(\frac{1}{\log{n_k}})$.
\end{theorem} 

As in the case of the diameter of random graphs (see e.g.~\cite[Theorem 10.10]{Bollobas}) one cannot expect to improve the $\omega(\frac{1}{\log{n_k}})$ term. Note however that for achieving diameter greater than~2 we needed an $\omega(\frac{\log{\log{n_k}}}{\log{n_k}})$ term rather than an $\omega(\frac{1}{\log{n_k}})$ term. It would be interesting to check whether this can be improved.

It would be also interesting to check whether loops in $\Gamma_x$ really do make a difference or whether the same results hold if $f_G(c)$ is replaced by $\tilde{f}_G(c) = \sum_{x \in G\setminus \{1\}}g_G(x)^c$. 

Finally we mention a related group theoretic problem that arises. As we have already mentioned the problem of which numbers in $(1/2,1]$ can appear as proportions of involutions of a finite group has been well studied and the groups with these proportions have been characterised. But what about numbers in the interval $[0,1/2]$? Is it true that for every number in this interval there are finite groups whose proportion of involutions is sufficiently close to this number, or are other gaps in this interval as well? We haven't been able to locate this problem in the literature but we believe it is an interesting one to investigate.

\bigskip

\noindent
{\footnotesize
Demetres Christofides, Queen Mary, University of London, School of Mathematical Sciences, London E1 4NS, United Kingdom, \href{mailto:christofidesdemetres@gmail.com}{\tt christofidesdemetres@gmail.com}

\smallskip

\noindent Klas \Markstrom, Department of Mathematics and Mathematical Statistics, \Umea\ University, 90187 \Umea, Sweden, \href{mailto:klas.markstrom@math.umu.se}{\tt klas.markstrom@math.umu.se}
}

\end{document}